%% file: mainJournal.tex
\documentclass{scrartcl}
\usepackage[utf8]{inputenc}
\RequirePackage[T1]{fontenc}
\usepackage{lmodern}

\usepackage{amssymb}
\usepackage{microtype}
\usepackage{amsmath}
\usepackage{amsfonts}
\usepackage{mathtools}

\usepackage{amsthm}

\usepackage{tikz}
\usetikzlibrary{calc}
\usetikzlibrary{hobby}
\usepackage{xcolor}
\pgfdeclarelayer{deepbackground}
\pgfdeclarelayer{background}
\pgfdeclarelayer{foreground}
\pgfsetlayers{deepbackground,background,main,foreground}

 \usepackage{stmaryrd}
 \usepackage{hyperref}
 \usepackage{cleveref}

\usepackage{nicefrac}
\usepackage{textpos}
\usepackage{thm-restate}

\usepackage{graphicx}
\usepackage{titling}
\usepackage{xspace}
\usepackage{macros}
\usepackage{tikzmacros}
\usepackage{Macros_journal}

\newcommand{\asdim}{\ensuremath{\variablestyle{asdim}}}
\newcommand{\andim}{\ensuremath{\variablestyle{ANdim}}}
\newcommand{\dsg}{$d$-sphere graph}
\newcommand{\BS}{\ensuremath{\mathbb{S}}}
\newcommand{\cc}{\ensuremath{\mathcal{C}}}
\newcommand{\ch}{\ensuremath{\mathcal{H}}}

\hypersetup{
colorlinks=true,
linkcolor=AO!65!black,
citecolor=AO!65!black,
urlcolor=AppleGreen!65!black,
bookmarksopen=true,
bookmarksnumbered,
bookmarksopenlevel=2,
bookmarksdepth=3
}

\title{Strongly sublinear separators and bounded asymptotic dimension for sphere intersection graphs}

\predate{}
\date{}
\postdate{}

\preauthor{}
\DeclareRobustCommand{\authorthing}{
    \begin{center}
        James Davies\thanks{Trinity Hall, University of Cambridge, United Kingdom, 
        \href{mailto:jgd37@cam.ac.uk}{jgd37@cam.ac.uk}} \hspace{1cm}
        Agelos Georgakopoulos\thanks{University of Warwick, United Kingdom, \href{mailto:a.georgakopoulos@warwick.ac.uk}{a.georgakopoulos@warwick.ac.uk}. Supported by EPSRC grants EP/V048821/1 and EP/V009044/1.} \hspace{1cm}
	    Meike Hatzel\thanks{Discrete Mathematics Group, Institute for Basic Science (IBS), Daejeon, South Korea, \href{mailto:research@meikehatzel.com}{research@meikehatzel.com}. The research was supported by the Institute for Basic Science (IBS-R029-C1).} \hspace{1cm}
        Rose McCarty\thanks{Georgia Institute of Technology, United States, \href{mailto:rmccarty3@gatech.edu}{rmccarty3@gatech.edu}. Supported by the National Science Foundation under Grant No.~DMS-2202961.}
\end{center}}
\author{\authorthing}
\postauthor{}

\newif\ifcomment
\commentfalse
\commenttrue

\renewcommand{\epsilon}{\varepsilon}
\renewcommand{\hat}{\widehat}
\renewcommand{\tilde}{\widetilde}
\renewcommand{\phi}{\varphi}

\begin{document}

\maketitle

\begin{abstract}
    In this paper, we consider the class $\SphereInt{d}$ of sphere intersection graphs in $\R^d$ for $d \geq 2$.
    We show that for each integer $t$, the class of all graphs in $\SphereInt{d}$ that exclude $K_{t,t}$ as a subgraph has strongly sublinear separators.
    We also prove that $\SphereInt{d}$ has asymptotic dimension at most $2d+2$.
\end{abstract}

\section{Introduction} 
We write
$\SphereInt{d}$ for the class of intersection graphs of spheres in $\R^d$ for $d \geq 2$.
The \emph{sphere dimension} of a graph $G$ is the smallest integer so that $G$ is 
contained in $\SphereInt{d}$.
This notion of dimension is analogous to the \emph{boxicity}, \emph{cubicity}, or \emph{sphericity}\footnote{We note that sphericity is more restrictive than sphere dimension as it only considers unit spheres.} of a graph.
These ``dimensionality'' parameters were introduced by Roberts~\cite{Roberts}, further studied by Maehara~\cite{Maehara}, and have since become a topic of much research; see for instance~\cite{ChandranFS10, Esperet16, ScottWood20, Trotter88}. We formulate an equivalent definition of $\SphereInt{d}$ utilising an appealing application of stereographic projection, allowing us to think of $\SphereInt{d}$ as a higher-dimensional generalisation of the classical class of circle graphs. (See~\cite{BouchetCircleChar, circleMultInterval, NajiCircle} for some classical work on circle graphs and~\cite{circleSurvey, McCarty} for surveys.) It will, therefore, become natural to let $\SphereInt{1}$ denote the class of circle graphs, which then have sphere dimension 1. Thus, $\SphereInt{d}$ simultaneously generalises various well-studied graph classes.

Of particular note, graphs with low dimension often satisfy strong separator theorems along the lines of Lipton and Tarjan's~\cite{LiptonTarjan79} famous separator theorem for planar graphs. Formally, a \emph{balanced separator} in an $n$-vertex graph $G$ is a set $S\subseteq V(G)$ so that each component of $G-S$ has at most $\frac{2}{3}n$ many vertices\footnote{The exact constant does not matter much; we could replace $\frac{2}{3}$ by any constant $c \in (0,1)$.}. Lipton and Tarjan showed that every $n$-vertex planar graph has a balanced separator of size $\mathcal{O}(\sqrt{n})$. These separator theorems have many applications 
as they can be combined with divide-and-conquer techniques~\cite{Chung1988, ETHtightEuclideanTSP, LiptonTarjan1980App, SmithWormald}.

Planar graphs are intersection graphs of internally disjoint circles by the Circle Packing Theorem of Andreev~\cite{Andreev}, Koebe~\cite{Koebe}, and Thurston~\cite{Thurston}. So Miller, Teng, Thurston, and Vavasis~\cite{MiTeThVaGeo} generalised the planar separator theorem by proving that the intersection graph of any sphere packing of size $n$ in $\R^d$ has a balanced separator of size $\mathcal{O}(n^{1-\frac{1}{d}})$. In fact, their results hold for any collection of balls of bounded ply; the \emph{ply} of a collection of geometric objects is the maximum number of objects with a non-empty common intersection. The ply of a collection of balls in $\R^d$ is closely related to the clique number of its intersection graph.

Many other geometric generalisations of Lipton and Tarjan's theorem have been proven, including some very recently~\cite{cliqueBasedSeps, cliqueBasedSepsStronger, Dvorak2021SSSconvexshapes, HarPeledQuanrud, separatorHyperbolic, SmithWormald}. Some of these results, such as~\cite{cliqueBasedSeps, cliqueBasedSepsStronger, separatorHyperbolic}, consider relaxations of separators and allow for more general geometric objects that way. However, if we keep the original notion of balanced separators, then two main improvements have been made. First, instead of considering balls, we can consider any compact convex subsets of $\R^d$ of bounded aspect ratio (the ratio of their diameter to their height) by~\cite[Lemma~2.21]{HarPeledQuanrud}. Second, it is enough that for any pair of objects under consideration, we can ``move one of them around inside the other to trace the boundary''; see~\cite[Theorem~2]{Dvorak2021SSSconvexshapes}.

All the results discussed above are for convex subsets of $\R^d$. We prove a new separator theorem for the intersection graphs of spheres in $\R^d$. Instead of bounding the ply, there is a different natural obstruction which we must bound the size of: 
the class of all complete bipartite graphs has bounded sphere dimension, but does not admit small balanced separators. 
Our first result is that this is the only obstruction. For any integers $t,d \geq 2$, we write $\SphereInt{d}_t$ for the class of all graphs that have sphere dimension at most $d$ and do not contain the complete bipartite graph $K_{t,t}$ as a subgraph. (This subgraph  need not be induced.)

\begin{restatable}{theorem}{StronglySublinearSeparatorsTheorem}
    \label{thm seps}
    For any integers $t,d \geq 2$, the class $\SphereInt{d}_t$ has strongly sublinear separators.
    More precisely, every $n$-vertex graph $G \in \SphereInt{d}_t$ has a balanced separator of size $\bigOlogd{t^{10}n^{1-\frac{1}{2d+8}}}$.
\end{restatable}

To the best of our knowledge, this is the first high-dimensional separator theorem that applies to classes that contain large induced complete bipartite graphs. In two dimensions, Lee~\cite{LeeSeparators} proved that every string graph with $m$ edges has a balanced separator of size $\mathcal{O}(\sqrt{m})$. This implies a separator theorem for string graphs with a forbidden $K_{t,t}$-subgraph by the K\H{o}v\'{a}ri-S\'{o}s-Tur\'{a}n Theorem~\cite{KST}. This theorem of Lee was conjectured by Fox and Pach~\cite{foxPach2010}, and a slightly weaker version with an extra $\log(m)$ factor was proven by Matou\v{s}ek~\cite{MatousekSeparators}. As far as we know, the same bound could hold for graphs with sphere dimension at most $d$; maybe every graph with $m$ edges and sphere dimension $d$ has a balanced separator of size $\mathcal{O}_d(\sqrt{m})$. If so, a theorem of Fox, Pach, Sheffer, Suk, and Zahl~\cite{SemiAlgebraic} about semi-algebraic graphs could then be used to obtain a strong separator theorem for $\SphereInt{d}_t$.

\medskip
Next, we turn our attention to the asymptotic dimension of our graphs. Gromov~\cite{GroAsyInv} introduced the notion of asymptotic dimension of a metric space as a coarse version of the classical topological dimension. This was part of his coarse-geometric approach, which has tremendously impacted geometric group theory. For instance, Yu~\cite{YuNov} showed that (any Cayley graph of) every finitely generated group $G$ with finite asymptotic dimension embeds coarsely into Hilbert space, and therefore $G$ satisfies a seminal conjecture of Novikov. Roughly speaking, asymptotic dimension corresponds to the number of colours needed to colour a graph, or family of graphs, with certain restrictions on the diameters and spacing of monochromatic components; the precise definition is given in \cref{sec:prelim}.

Although asymptotic dimension was originally introduced to study spaces of infinite diameter, the notion can be applied to classes of finite graphs, and it is gaining popularity in this setting \cite{bonamy2023asymptotic,DisPro,Dvorak2021SSSconvexshapes,FujPapAsy,GeoPapMin,JorLanGeo,LiuAss}. This popularity is partly motivated by a connection to sparse partition schemes in algorithm development, as we discuss in \cref{sec probs}. It is also motivated by connections to graph colouring, as discussed in~\cite{WeakDiamColoringSurfaces}.

Thus, understanding how the asymptotic dimension of a graph class interacts with other geometric properties is important. In this spirit, Bonamy et al.~\cite{bonamy2023asymptotic} proved that planar graphs have asymptotic dimension at most 2, and asked if every class with strongly sublinear separators has bounded asymptotic dimension~\cite{bonamy2023asymptotic}. (We note that this question was phrased in terms of classes of polynomial expansion; however, a class has polynomial expansion if and only if it has strongly sublinear separators~\cite{2016sssPolyExp, PRS}. Formally, a class has \emph{strongly sublinear separators} if there exists $\delta>0$ so that every $n$-vertex graph that is an induced subgraph of some graph in the class admits a balanced separator of size $\mathcal{O}(n^{1-\delta})$.) We prove that 
the answer is positive for graphs of bounded sphere dimension.

\begin{theorem}
For any integer $d\geq 1$, the class of graphs of sphere dimension at most $d$ has asymptotic dimension at most $2d+2$.
\end{theorem}

This theorem generalises and also uses a theorem of Dvo{\v{r}}{\'a}k and Norin~\cite{dvovrak2022asymptotic}, at the cost of increasing the dimension by 1. In fact, we prove a more general result by considering intersection graphs of shapes other than spheres (\cref{thm asdim}).

\paragraph*{Paper structure.}
We introduce some preliminaries in \cref{sec:prelim}.
In~\cref{sec SG}, we discuss basic properties of sphere dimension, particularly the connection to circle graphs. 
In~\cref{sec:SSS}, we prove our main result that $\SphereInt{d}_t$ has strongly sublinear separators.
We prove that $\SphereInt{d}$ has bounded asymptotic dimension in \cref{sec:asymp_dim}.
We conclude with some open problems in~\cref{sec probs}.

\section{Preliminaries}
\label{sec:prelim}

All graphs in this paper can be assumed to be finite and simple, although some of our results extend verbatim to the infinite case. 
We use $\bigOlogd{}$ for the version of big-$\bigO{}$ notation where we view $d$ as being fixed and suppress polylogarithmic factors. (In fact it will be clear from our proof of \cref{thm seps} that we only suppress a single $\log{n}$ factor and a constant factor of the form $c^d$ for $c$ a universal constant.)

\paragraph*{Geometric intersection graphs.}
Let $\mathcal{G}$ be a class of geometric objects in a given space, for example, unit squares in the plane. Then an \emph{intersection graph of $\mathcal{G}$} is a graph $G$ whose vertices can be identified with objects in $\mathcal{G}$ such that there is an edge between two vertices in $G$ if and only if the two corresponding objects intersect. Similarly, we write $G(\mathcal{G})$ for the graph with vertex-set $\mathcal{G}$ where two vertices are adjacent if they have a non-empty intersection. The \emph{ply} of $\mathcal{G}$ is the maximum size of a set of objects in $\mathcal{G}$ with a non-empty common intersection. 

\paragraph*{Shallow minors.}
Our proof of \cref{thm seps} relies on a theorem of Plotkin, Rao, and Smith~\cite{PRS} and Dvo{\v{r}}{\'a}k and Norin~\cite{2016sssPolyExp} which characterises classes that admit strongly sublinear separators in terms of forbidden shallow minors. We give some relevant definitions.

Given a positive integer $r$, we say that a graph $H$ is an \emph{$r$-shallow minor} of a graph $G$ if there is a function $\phi: \V{H} \to 2^{\V{G}}$ such that:
\begin{itemize}
    \item for every $x \in V(H)$, the subgraph of $G$ induced by $\phi(x)$ contains a rooted spanning tree of height at most $r$,
    \item for all distinct $x,y \in V(H)$, the sets $\phi(x)$ and $\phi(y)$ are disjoint, and
    \item for every $xy \in \E{H}$, there is an edge in $G$ with one end in $\phi(x)$ and the other in $\phi(y)$.
\end{itemize}
We call $\phi$ an \emph{$r$-shallow minor model} of $H$ in $G$. For $v \in V(H)$, we call $\phi(v)$ the \emph{bag of $v$}. 

\paragraph*{Strongly sublinear separators.}
Recall that a \emph{balanced separator} in an $n$-vertex graph $G$ is a set $S \subseteq V(G)$ so that each component of $G-S$ has at most $\frac{2}{3}n$ vertices.
A graph class $\mathcal{C}$ has \emph{strongly sublinear separators} if there exist $\epsilon>0$ so that for any $n$-vertex graph $H$ which is an induced subgraph of a graph in $\mathcal{C}$, the graph $H$ has a balanced separator of size $\bigO{n^{1-\epsilon}}$. (The big-$\bigO{}$ notation can hide a constant factor that depends on $\mathcal{C}$.)
Note that to obtain strongly sublinear separators, it suffices to show that there exists $\delta>0$ so that there are balanced separators of size $\bigOlog{n^{1-\delta}}$.
In order to provide this, we show that there exists $\delta>0$ so that there are separators of size $\bigOlog{n^{1-\delta}}$.
Plotkin, Rao, and Smith~\cite{PRS} proved that any graph with a forbidden shallow minor has a small balanced separator.

\begin{theorem}[Plotkin, Rao, and Smith~\cite{PRS}]
    \label{thm:PRS}
    For any $r,h \in \N$, any $n$-vertex graph without an $r$-shallow $K_h$ minor has a balanced separator of size $\PRSfunction{r}{h}$.
\end{theorem}

We apply this theorem with $r$ being a small power of $n$ and $h$ being a polynomial in $r$.
With this choice of parameters, \cref{thm:PRS} yields a sufficiently small balanced separator.

\paragraph*{Strong colouring number.}
Let $G$ be a graph, $r \in \N$, and $\leq$ be an ordering of the vertices of $G$.
For a vertex $v \in \V{G}$, we call a vertex $u \in \V{G}$ \emph{strongly $r$-reachable from $v$} if $u \leq v$ and there exists a $v$-$u$-path $P$ of length at most $r$ in $G$ such that $u$ is the only vertex in $P$ that is smaller than $v$ with respect to $\leq$.
The \emph{$r$-width} of $\leq$ is the maximum over all $v \in V(G)$ of the number of vertices which are strongly $r$-reachable from $v$.
Finally, the \emph{$r$-strong colouring number} of a graph $G$, denoted $\scol{r}{G}$, is the minimum $r$-width over all possible vertex orderings of $G$.
We require the observation that intersection graphs of balls of bounded ply have bounded strong colouring numbers.

\begin{lemma}[{\cite[Lemma 1]{2022plyballscol}}]
    \label{lem:ball-graphs-ply-scol}
    Let $\mathcal{B}$ be a finite collection of balls in $\R^d$ of ply at most $t$. Then the intersection graph $G(\mathcal{B})$ has $r$-strong colouring number at most $t(2r+2)^d$.
\end{lemma}

We also use the following well-known observation that graphs with large shallow clique minors have large strong colouring numbers.
We give a proof for the sake of completeness.

\begin{observation}
    \label{obs:shallow-minor_implies_large_scol}
    If a graph $G$ contains an $r$-shallow $K_h$-minor, then $\scol{4r+1}{G} \geq h-1$.
\end{observation}
\begin{proof}
    Let $\leq$ be an ordering of minimal $(4r+1)$-width for a graph $G$ containing an $r$-shallow $K_h$-minor.
    Let $\phi$ be an $r$-shallow minor model of $K_h$ in $G$. We denote $V(K_h) = \{v_1, v_2, \ldots, v_h\}$ for convenience.

    We choose $x \in \V{G}$ maximal in $\leq$ such that the vertex set $X \coloneqq \Set{y \mid x \leq y}$ contains a bag of $\phi$.
    That is, there exists $v_i \in V(K_h)$ so that $\phi(v_i) \subseteq X$. Without loss of generality, we may assume that $\phi(v_1) \subseteq X$. As $x$ is chosen maximal, the bag of any other vertex $v_i$ of $K_h$ contains a vertex $y_i$ with $y_i \leq x$.
    As $K_h$ is the complete graph on $h$ vertices, there is a path $P_i$ in $G[\phi(v_1)\cup \phi(v_i)]$ from $x$ to $y$. Because $\phi$ is an $r$-shallow minor model, we may choose $P_i$ to have length at most $4r+1$.
    Now, for all $2\leq i \leq h$, let $P'_i$ be the path obtained from $P_i$ by beginning at $x$ and ending at the first vertex which is strictly less than $x$ according to $\leq$. As all the bags are disjoint, these paths $P'_2,\dots, P'_h$ end at pairwise distinct vertices.
    This shows that $\leq$ has $(4r+1)$-width at least $h-1.$
\end{proof}

\newcommand{\polySCOLSSSfunction}[2]{\bigOlog{n^{1-\frac{1}{2#2+2}}}}
\newcommand{\polySCOLSSScdeltafkt}[2]{\bigOlog{n^{1-\frac{1}{2#2+2}}}}

\paragraph*{Asymptotic and Assouad--Nagata dimension.}
Let $G$ be a graph, and let $\dist{G}{\cdot,\cdot}$ denote its usual graph metric.
If the graph is clear from the context, we drop the index.
Let $\mathcal{U}$ be a family of subsets of $V(G)$.
We say that $\mathcal{U}$ is \emph{$D$-bounded} if each set $U\in \mathcal{U}$ has diameter at most $D$.
We say that $\mathcal{U}$ is \emph{$r$-disjoint} if for any $a, b$ belonging to different elements of $\mathcal{U}$ we have $\dist{}{a,b}> r$.

We say that $D: \R_+ \to \R_+$ is an \emph{n-dimensional control function} for a graph class $\cc$, if for any $G\in \cc$ and $r > 0$,  $V(G)$ has a cover $\mathcal{U} = \bigcup_{i=1}^{n+1}  \mathcal{U}_i$  such that each $\mathcal{U}_i$ is $r$-disjoint and each element of $\mathcal{U}$ is $D(r)$-bounded. The \emph{asymptotic dimension} of  $\cc$ denoted by $\asdim(\cc)$,
is the least integer $n$ such that $\cc$ has an $n$-dimensional control function. If no such integer $n$ exists, then $\asdim(\cc)$ is infinite. The \emph{Assouad--Nagata dimension} of $\cc$, denoted by $\andim(\cc)$, is defined similarly, except that we insist that $D(r) \in \bigO{r}$.

\section{Sphere graphs and sphere dimension} \label{sec SG}

In this section, we show that the graph classes $\SphereInt{d}$ can be thought of as a higher-dimensional analogue of the well-studied class of circle graphs. 
Recall that a \emph{circle graph} is an intersection graph of a family of chords of $\BS^1$. We generalise this definition as follows. A 
\emph{\dsg} is an intersection graph of a subfamily of 
\begin{equation*}
    \ch_d \coloneqq \{ \mathbb{D}^d \cap H \mid H \text{ is a hyperplane of $\R^{d+1}$}\},
\end{equation*}
where the $d$-disc $\mathbb{D}^d$ is the set $\{x\in \R^{d+1} \mid |x,0|\leq 1\}$, and $\BS^d$ is its boundary.
In particular, $\ch_1$ consists of all chords of the circle $\BS^1$, and $\ch_2$ consists of the flat discs with boundary on $\BS^2$.
Thus, 1-sphere graphs are exactly the circle graphs, which by definition are the graphs in $\SphereInt{1}$.
We observe that 
the class of \dsg s coincides with $\SphereInt{d}$: 

\begin{observation} \label{obs SP}
For every $d\geq 1$, a countable graph $G$ is a $d$-sphere graph if and only if it belongs to $\SphereInt{d}$.
\end{observation}
\begin{proof}
For $d=1$ this is so by definition. For $d\geq 2$, suppose $\Theta= \{H_v \mid v\in V(G)\}$ is a $d$-sphere representation of $G$, i.e.\ a family of elements of $\ch_d$ witnessing that $G$ is a \dsg. Assume that no $H_v$ contains the \emph{north pole} $\{0,\ldots,0,1\}$ of $\BS^d$, which we can since $G$ is countable (which is a much stronger assumption than we actually need). Notice that $H_v \cap \BS^d$ is homeomorphic to $\BS^{d-1}$ for every $v$. Letting $f$ be the stereographic projection of $\BS^d$ onto $\R^d$, we observe that $f(H_v \cap \BS^d)\eqqcolon H'_v$ is also homeomorphic to $\BS^{d-1}$, where we use the well-known property of stereographic projection that it preserves spheres not containing the north pole \cite[Lemma~4.2]{MiTeThVaGeo}. Then $G$ is isomorphic to the intersection graph of the family $\Theta'= \{H'_v \mid v\in V(G)\}$ by construction.

Conversely, we use $f^{-1}$ to stereographically project a representation of $G$ as an intersection graph of spheres in $\R^d$ (avoiding the origin) to a $d$-sphere graph representation.
\end{proof}

We now give some other observations about sphere graphs and sphere dimension. As a warm-up, we observe that if $G$ is planar, then it has sphere dimension~$2$. This is because the Circle Packing Theorem \cite{Koebe} asserts that every finite\footnote{More generally, this applies to any connected, locally finite, planar graph; see \cite{HeSchrFix} and references therein.} planar graph admits a sphere packing in~$\R^2$. (A \emph{sphere packing} of a graph $G$ is an arrangement of closed euclidean spheres $\{S_v: v\in V(G)\}$ in $\R^d$ or $\BS^d$ \st $S_v \cap S_w$ is empty whenever $vw\not\in E(G)$ and $S_v \cap S_w$ is a single point whenever $vw\in E(G)$. In other words, $G$ is the tangency graph of the family $\{S_v:v \in V(G)\}$.) 
It is known that every finite graph $G$ admits a sphere packing in $\R^d$ for $d\leq |V(G)|-1$ \cite{MOEppstein}, and so $G \in \SphereInt{|V(G)|-1}$. This proves the following observation.


\begin{observation} \label{obs every}
Every graph $G$ with $|V(G)|\geq 3$ has sphere dimension at most ${|V(G)|-1}$.
\end{observation}
%

Conversely, we prove the following.

\begin{observation} \label{thm not d}
For every $d\in \N$, there is a finite graph which is not a $d$-sphere graph.
\end{observation}
\begin{proof}
We may assume that $d\geq 2$ since it is well known that not every graph is a circle graph; see~\cite{BouchetCircleChar}. Let $n$ be an integer; we think of $n$ as being large in comparison to $d$.

Let $G_n$ be a graph of minimum degree $n$ that does not contain a triangle as a subgraph, and such that $G_n$ remains connected after removing the closed neighbourhood $N[y]\coloneqq N(y) \cup \{y\}$ of any $y\in V(G_n)$; for example, $G_n$ could be a portion of the $n$-dimensional cubic lattice. Suppose $G_n$ can be realised as a \dsg, and therefore as an intersection graph of spheres $\{S_v \mid v\in V(G_n)\}$ in $\R^d$ by \cref{obs SP}. Pick a sphere of minimum radius among the $S_v$. Then, the spheres representing the neighbours of $v$ intersect $S_v$ and are mutually disjoint since $G_n$ is triangle-free. Moreover, there is no triple $S_x, S_y, S_z$ of such spheres \st $S_y$ lies in the interior of $S_x$ and $S_z$ lies in the interior of $S_y$ because $N[y]$ would separate $x$ from $z$ in this case. Thus, after deleting any such spheres lying in the interior of others, we are left with at least $n/2$ spheres with disjoint interiors intersecting $S_v$, each of radius at least that of $S_v$. By a volume argument, the number of such spheres is bounded by a function of $d$. Thus $G_n$ is not a \dsg\ when $n$ is large enough.
\end{proof}

It is possible to strengthen \cref{thm not d} by providing cubic graphs of arbitrarily high sphere dimension. Indeed, this is a corollary of \cref{thm seps}, as we now discuss. Let $\{G^d_n: n\in \N\}$ be a family of cubic expander graphs with $\lim_{n\to \infty} |V(G^d_n)| = \infty$. Then, since expanders do not have small separators, at most finitely many members of this family can have sphere dimension at most $d$ by \cref{thm seps} and by Lee's~\cite{LeeSeparators} separator theorem for circle graphs. Alternatively, we can let $\{G^d_n: n\in \N\}$ be a family of finite cubic graphs of asymptotic dimension more than $2d+1$; for instance, such graphs can be obtained from a portion of the $(2d+2)$--dimensional hypercubic lattice by replacing each vertex with a subcubic tree. Then at most finitely many members of this family can have sphere dimension at most $d$ by the result of \cite{dvovrak2022asymptotic} that our \cref{thm asdim} generalises.

\section{Strongly sublinear separators}
\label{sec:SSS}

In this section, we prove \cref{thm seps}, the separator theorem for the class $\SphereInt{d}_t$ of all graphs which have sphere dimension at most $d$ and do not have $K_{t,t}$ as a subgraph. 

We need some additional notation. Given a sphere $S$ in $\R^d$, we write $\ball{S}$ for the (closed) ball with boundary $S$.
We say that a sphere $S'$ \emph{contains} a different sphere $S$ if $\ball{S'}$ contains $\ball{S}$.
Given a collection $\mathcal{S}$ of spheres in $\R^d$, we say that a sphere $S \in \mathcal{S}$ is \emph{maximal} if there is no $S' \in \mathcal{S}\setminus \{S\}$ such that $S'$ contains $S$. Likewise, a sphere $S \in \mathcal{S}$ is \emph{minimal} if there is no $S' \in \mathcal{S}\setminus \{S\}$ such that $S$ contains $S'$.
We call $\mathcal{S}$ \emph{nested} if for any two spheres in $\mathcal{S}$, one contains the other. In this section, every collection of spheres is finite.

We start by noting that excluding $K_{t,t}$ provides a bound on the number of {short} paths between the ``inside'' and the ``outside'' of a large set of nested spheres.

\begin{lemma}
    \label{lem:few_short_paths_through_nested_spheres}
    Let $r,t \in \N$, let $\mathcal{S}$ be a collection of spheres in $\R^d$, and let $\mathcal{N} \subseteq \mathcal{S}$ be a set of $t(r+1)$-many nested spheres. Suppose that the intersection graph $G(\mathcal{S}\setminus \mathcal{N})$ contains a collection $\mathcal{P}$ of $t^2(r+1)$-many pairwise vertex-disjoint paths such that for each path $P \in \mathcal{P}$,
    \begin{itemize}
        \item one end of $P$ intersects the minimal sphere in $\mathcal{N}$, 
        \item the other end of $P$ intersects the maximal sphere in $\mathcal{N}$, and
        \item $P$ has length at most $r$.
    \end{itemize}
    Then, $G(\mathcal{S})$ contains a $K_{t,t}$ as a subgraph.
\end{lemma}
\begin{proof}
    We sort the vertices in $\mathcal{N}$ by containment and consider $\leq$ to be this ordering. For every path $P \in \mathcal{P}$ and every vertex $x \in \V{P}$, the neighbourhood of $x$ in $\mathcal{N}$ is an interval with respect to $\leq$.
    We denote this interval by $I_x$.
    The first two conditions of the \namecref{lem:few_short_paths_through_nested_spheres} guarantee that every sphere in $\mathcal{N}$ is a neighbour of at least one vertex in $P$. Thus, since $\Abs{\mathcal{N}} = t(r+1)$ and $P$ has length at most $r$, every $P \in \mathcal{P}$ contains a vertex $x_P$ with $\Abs{I_{x_P}} \geq t$.
    Also, there are at most $t(r+1)$ choices for the first vertex (according to $\leq$) inside $I_{x_P}$.
    Therefore, there are $t$ vertices of the form $x_P$ (on different paths) such that their intervals all start in the same vertex.
    These $t$ vertices, plus the first $t$ vertices in their intervals, yield the desired $K_{t,t}$-subgraph.
\end{proof}

Using this insight, we next prove that ``minimal'' graphs that contain a certain shallow clique minor cannot be represented by spheres with too much nesting.

\begin{lemma}
    \label{lem:minimal_shallow_model_bounded_nestedness}
    Let $t,d,r,h \in \N$ with $t,d \geq 2$, and suppose that $G \in \SphereInt{d}_t$ has as few vertices as possible subject to containing an $r$-shallow $K_h$-minor. Let $\mathcal{S}$ be a collection of spheres in $\R^d$ such that $G=G(\mathcal{S})$. Then every nested subset of $\mathcal{S}$ has size at most $36t^4(3r+1)^3$.
\end{lemma}

\begin{proof}
    Suppose towards a contradiction that there is a nested set $\mathcal{N}\subseteq \mathcal{S}$ of size at least $36t^4(3r+1)^3$.
    Let $\phi$ be an $r$-shallow minor model of $K_h$ in $G$. For each vertex $x \in V(K_h)$, we fix some rooted spanning tree $T_{x}$ of the subgraph of $G$ induced on $\phi(x)$ of height at most $r$.
    Note that, as $G$ is minimal with respect to containing an $r$-shallow $K_h$-minor, every vertex of $G$ is in a bag of $\phi$. Similarly, for every vertex $x \in V(K_h)$, every leaf of $T_{x}$ has an edge to some other bag which has no other neighbours in $T_{x}$.
    We now break into cases based on whether many vertices in $\mathcal{N}$ lie in the same bag or in different bags.
    
    First, consider the case that at least $6t^2(3r+1)^2$ vertices of $\mathcal{N}$ lie in a single bag, say the bag of $x \in V(K_h)$. As $T_{x}$ has height at most $r$, at least a $\frac{1}{(r+1)}$ proportion of these elements all have the same distance to the root of $T_{x}$. Thus, there exists a set $\mathcal{X} \subseteq \mathcal{N}$ of size at least $6t^2(3r+1)$ such that all vertices in $\mathcal{X}$ are not only in $T_{x}$ but also have the same distance to the root of $T_{x}$. So, no vertex in $\mathcal{X}$ lies on the path between the root of $T_x$ and another vertex in $\mathcal{X}$. For each $X \in \mathcal{X}$, we fix an arbitrary leaf vertex $\ell(X)$ of $T_x$ that is a descendent of $X$ in $T_x$. (It is possible that $\ell(X)=X$.) Let us denote the path of $T_x$ joining $X$ and $\ell(X)$ by $P(X)$. Since all of the vertices in $\mathcal{X}$ have the same distance to the root, the paths $\{P(X) \mid X \in \mathcal{X}\}$ are pairwise vertex-disjoint. Additionally, for each of these leaves $\ell(X)$, there is a vertex $y(X) \in \V{K_h}$ such the bag of $y(X)$ has an edge to $\ell(X)$ and does not have an edge to any other vertex in $\phi(x)$. Thus, since the leaves $\{\ell(X) \mid X \in \mathcal{X}\}$ are all distinct, the vertices $\{y(X) \mid X \in \mathcal{X}\}$ are also all distinct.

    Now we sort the elements of $\mathcal{X}$ by containment as $X_1,\dots, X_{\Abs{\mathcal{X}}}$.
    We pair up the first $t^2(6r+2)$ elements of $\mathcal{X}$ with the last $t^2(6r+2)$ elements of $\mathcal{X}$.
    So we consider the pairs $(X_1,X_{\Abs{\mathcal{X}}})$, $(X_2,X_{\Abs{\mathcal{X}}-1})$, and so on.
    As $K_h$ is the complete graph, for every pair $(X_i,X_{\Abs{\mathcal{X}}-i+1})$ there is a path between $X_i$ and $X_{\Abs{\mathcal{X}}-i+1}$ whose vertex-set is contained in $P(X_i) \cup \phi(y(X_i)) \cup \phi(y(X_{\Abs{\mathcal{X}}-i+1})) \cup P(X_{\Abs{\mathcal{X}}-i+1})$.
    The two paths $P(X_i)$ and $P(X_{\Abs{\mathcal{X}}-i+1})$ have length at most $r-1$ since no vertex in $\mathcal{X}$ is the root of $T_x$. In each of the two bags $\phi(y(X_i))$ and $\phi(y(X_{\Abs{\mathcal{X}}-i+1}))$, we can choose paths of length at most $2r$.
    Finally, we need $3$ edges to join the parts together.
    So overall, this path between $X_i$ and $X_{\Abs{\mathcal{X}}-i+1}$ has length at most $2(r-1)+2(2r)+3=6r+1$ (see~\cref{fig:case_1}).
    Note that $\Abs{\mathcal{X}} \geq 6t^2(3r+1)= 3t^2(6r+2)\geq 2t^2(6r+2)+t(6r+2)$.
    So at least $t(6r+2)$ spheres lie between any two paired spheres in $\mathcal{X}$.
    Thus, by \cref{lem:few_short_paths_through_nested_spheres}, $G$ contains $K_{t,t}$ as a subgraph, a contradiction.
    
    \begin{figure}[!ht]
        \begin{center}
        \input{figures/case-1-figure.tex}
        \end{center}
        \caption{This illustrates the case that many vertices of $\mathcal{N}$ lie in a single bag.
        The set \textcolor{LavenderMagenta}{$\mathcal{X}=\Set{X_1,\dots, X_6}$} is depicted together with the corresponding \textcolor{CornflowerBlue}{leaves} (note that $X_3$ is equal to its leaf).
        Using two paired bags, the \textcolor{AppleGreen}{green paths} are of length at most $6r+1$ and pairwise disjoint.
        }
        \label{fig:case_1}    
    \end{figure}
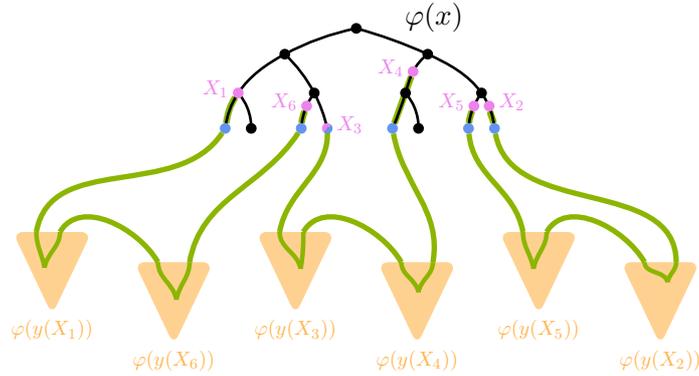

     \begin{figure}[!ht]
        \begin{center}
        \input{figures/case-2-figure.tex}
        \end{center}
        \caption{This illustrates the case  that many vertices of $\mathcal{N}$ lie in different bags.
        The \textcolor{LavenderMagenta}{outermost spheres} are mapped to the \textcolor{CornflowerBlue}{innermost spheres} such that the \textcolor{AppleGreen}{green paths} (all of length at most $4r+2$) have to cross the spheres lying in between.}
        \label{fig:case_2}
    \end{figure}
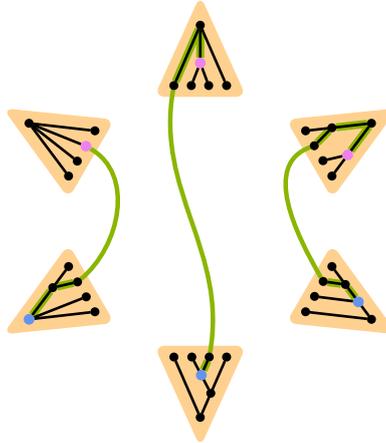

    Second, consider the case that there exists a set $\mathcal{X}\subseteq \mathcal{N}$ of size at least $6t^2(3r+1)$ so that every pair of vertices in $\mathcal{X}$ lie in different bags of $\phi$.
    So for every $X \in \mathcal{X}$, there exists a vertex $y(X) \in \V{K_h}$ such that $X$ lies in $\phi(y(X))$.
    Again, we sort the elements of $\mathcal{X}$ by containment as $X_1,\dots, X_{\Abs{\mathcal{X}}}$ and pair the first and last elements.
    This time, we pair the first $t(4r+2)$ elements of $\mathcal{X}$ with the last $t(4r+2)$ elements of $\mathcal{X}$.
    So for $1 \leq i \leq t(4r+2)$, the sphere $X_i$ is paired to the sphere $X_{\Abs{\mathcal{X}}-i+1}$.
    As $K_h$ is the complete graph, there exists a path between $X_i$ and $X_{\Abs{\mathcal{X}}-i+1}$ of length at most $4r+1$ in the subgraph of $G$ induced by $\phi(y(X_i)) \cup \phi(y(X_{\Abs{\mathcal{X}}-i+1}))$ (see \cref{fig:case_2}).
    Additionally, $\Abs{\mathcal{X}}\geq 6t^2(3r+1)= 3t^2(6r+2) \geq 2t(4r+2)+t^2(4r+2)$. 
    Thus, at least $t^2(4r+2)$ spheres lie between any two paired spheres in $\mathcal{X}$.
    Thus, again by \cref{lem:few_short_paths_through_nested_spheres}, $G$ contains $K_{t,t}$ as a subgraph, a contradiction.

    One of the two cases must occur, and so this completes the proof of \cref{lem:minimal_shallow_model_bounded_nestedness}.
\end{proof}

It is known, as seen in \cref{lem:ball-graphs-ply-scol}, that intersection graphs of balls of small ply have bounded strong colouring number.
We make use of this to show that we can obtain a similar result for sphere intersection graphs excluding a $K_{t,t}$-subgraph, assuming they can be represented without a large set of nested spheres.

\begin{lemma}
    \label{lem:bounded_nestedness_bounded_scol}
    Let $\mathcal{S}$ be a collection of spheres in $\R^d$ whose intersection graph $G=G(\mathcal{S})$ does not contain $K_{t,t}$ as a subgraph. Suppose that every nested subset of $\mathcal{S}$ has size at most $k$. Then for every $r \in \N$, the $r$-strong colouring number of $G$ is at most $2kt(2r+2)^d$.
\end{lemma}
\begin{proof}
    We claim that the collection $\{\ball{S}\mid S \in \mathcal{S}\}$ of balls in $\R^d$ has ply at most $2kt$. This yields the desired statement by~\cref{lem:ball-graphs-ply-scol}.

    So, let $\mathcal{X} \subseteq \mathcal{S}$ be a collection of spheres such that their common intersection as balls is non-empty, that is, such that $\bigcap_{X \in \mathcal{X}} \ball{X} \neq \emptyset$.
    Let us consider the poset $\preceq$ on $\mathcal{X}$ where for any spheres $X, X' \in \mathcal{X}$, we have $X \preceq X'$ if $X$ is contained in $X'$.
    Every chain in this poset has size at most $k$ by the assumption about nested sets.
    Next, note that a single point in $\R^d$ lies in at most $2t$ pairwise incomparable (according to $\preceq$) spheres in $\mathcal{X}$, as otherwise, $G$ would contain a clique of size $2t$ and therefore have $K_{t,t}$ as a subgraph.
    Thus, every antichain in this poset has size at most $2t$, and the \namecref{lem:bounded_nestedness_bounded_scol} follows.
\end{proof}

The penultimate step is to combine \cref{lem:minimal_shallow_model_bounded_nestedness,lem:bounded_nestedness_bounded_scol} with \cref{obs:shallow-minor_implies_large_scol} in order to exclude shallow minors from the class $\SphereInt{d}_t$.

\begin{lemma}
\label{lem:forbidShallowMinor}
For any integers $t,d \geq 2$ and any $r \in \N$, no graph $G\in\SphereInt{d}_t$ has a clique of size $72t^5(3r+2)^{d+3}$ as an $r$-shallow minor.
\end{lemma}
    \begin{proof} For convenience, we set $h=72t^5(3r+2)^{d+3}$.
    Suppose for a contradiction that there is a graph $G \in \SphereInt{d}_t$ which contains $K_h$ as an $r$-shallow minor. Choose such a graph $G$ with as few vertices as possible. Let $\mathcal{S}$ be a collection of spheres in $\R^d$ such that $G = G(\mathcal{S})$.
    Then by \cref{lem:minimal_shallow_model_bounded_nestedness}, every nested subset of $\mathcal{S}$ has size at most $36t^4(3r+1)^3$.
    So by \cref{lem:bounded_nestedness_bounded_scol}, the $(4r+1)$-strong colouring number of $G$ is at most $2(36t^4(3r+1)^3)t(2r+2)^d$.
    Note that
    \begin{equation*}
        2(36t^4(3r+1)^3)t(2r+2)^d = 72t^5(3r+1)^3(2r+2)^d \leq 72t^5(3r+2)^{d+3}-2 = h-2.
    \end{equation*}
    
    Finally, by \cref{obs:shallow-minor_implies_large_scol}, the $(4r+1)$-strong colouring number of $G$ is at least $h-1$. So $h-1 \leq \scol{4r+1}{G}  \leq h-2$, a contradiction.
\end{proof}

Now we can combine our results with the theorem of Plotkin, Rao, and Smith~\cite{PRS} (stated as \cref{thm:PRS}) to obtain strongly sublinear separators in graphs of sphere dimension at most $d$ which exclude a $K_{t,t}$-subgraph. We restate the theorem below for convenience.

\StronglySublinearSeparatorsTheorem*
\begin{proof}
    Let $G \in \SphereInt{d}_t$. For convenience, set $r = n^{\frac{1}{2d+8}}$. (We omit floors and ceilings for the sake of readability when they do not affect the proof.) By \cref{lem:forbidShallowMinor}, $G$ does not have a clique of size $72t^5(3r+2)^{d+3}$ as an $r$-shallow minor. Since $72t^5(3r+2)^{d+3} \in \mathcal{O}_d(t^5r^{d+3})$, by \cref{thm:PRS} we have that $G$ contains a balanced separator of size
    \begin{equation*}
        \mathcal{O}_d\left(rt^{10}r^{2d+6}\log{n}+\frac{n}{r} \right) = \bigOlogd{}\left(t^{10}r^{2d+7}+n^{1-\frac{1}{2d+8}}\right) = \bigOlogd{}\left(t^{10}n^{1-\frac{1}{2d+8}}\right).\qedhere
    \end{equation*}
\end{proof}

\section{Asymptotic dimension}
\label{sec:asymp_dim}

\newcommand{\IntGraphs}[2]{\mathcal{C}^{#1}_{#2}} 
\newcommand{\WeightedIntGraphs}[2]{\hat{\mathcal{C}}^{#1}_{#2}} 
\newcommand{\CompIntGraphs}[2]{\mathcal{B}^{#1}_{#2}} 

In this section, we prove that $\SphereInt{d}$ has bounded asymptotic dimension. 
In fact, we prove that the following more general class of graphs has bounded asymptotic dimension.
Let $\IntGraphs{d}{\alpha}$ be the class of connected intersection graphs of systems of path-connected sets in $\R^d$, each of which is obtained from a compact convex set of aspect ratio at most $\alpha$ by removing some subset of its interior.
Clearly, the class $\SphereInt{d}$ of intersection graphs of spheres in $\R^d$ is a subclass of $\IntGraphs{d}{\alpha}$.
If the sets are restricted so that none of their interiors is removed, then this subclass was proved to have bounded asymptotic dimension by Dvo{\v{r}}{\'a}k and Norin \cite{dvovrak2022asymptotic}.
This theorem shall be one of our main tools. 

\begin{theorem}[\cite{dvovrak2022asymptotic}]\label{thm:convexasdim} 
    For every positive integer $d$ and every $\alpha \ge 1$, the class of intersection graphs of systems of compact convex sets of aspect ratio at most $\alpha$ in $\R^d$ has asymptotic dimension at most $2d + 1$.
\end{theorem}

We introduce some notations similar to those used for spheres as in \cref{sec:SSS}. 
For a graph $G\in \IntGraphs{d}{\alpha}$ and a vertex $x \in\V{G}$ we refer by $\closure{x}$ to the convex closure of $x$.
We say that a vertex $y$ \emph{contains} a different vertex $x$ if $\closure{y}$ contains $\closure{x}$.
We say the vertex $x$ is \emph{maximal} if there is no $y \in \V{G} \backslash \{x\}$ such that $y$ contains $x$.

The \emph{height} of a bounded convex set $X\subset \R^d$ is the infimum $h
\geq 0$ such that $X$ is contained between two parallel hyperplanes at distance $h$.
The \emph{aspect ratio} of $X$ is the ratio of the diameter of $X$ over its \emph{height}. Let $\CompIntGraphs{d}{\alpha}$ denote the class of intersection graphs of systems of compact convex sets of aspect ratio at most $\alpha$ in $\R^d$. 
If $G\in \IntGraphs{d}{\alpha}$ contains no pair of distinct vertices $x,y$ with $y$ containing $x$, then $G$ belongs to $\CompIntGraphs{d}{\alpha}$, i.e.~\cref{thm:convexasdim} applies. 

\medskip
Our next main tool is a theorem of Bonamy, Bousquet, Esperet, Groenland, Liu, Pirot, and Scott \cite{bonamy2023asymptotic} that shall allow us to apply \cref{thm:convexasdim}.
Before introducing this theorem, we require some further definitions.

A \emph{weighted graph} $(G, \weight)$ consists of a graph $G$ and a function $\weight: E(G) \to {\R}^+$.
We call $\weight(e)$ the \emph{weight} of $e$ for each $e \in E(G)$. The length in $(G, \weight)$ of a path $P$ in $G$ is the sum of the weights of the edges of $P$. Given two vertices $x, y \in V(G)$, we define $\dist{(G, \weight)}{x,y}$ to be the infimum of the length in $(G, \weight)$ of a path between $x$ and $y$; we define $\dist{(G, \weight)}{x,y} = \infty$ if there exists no path between $x$ and $y$.
We remark that all the weighted graphs we consider are actually graphs whose edges have weights $1$ or $2$.
Given $A\subseteq V(G)$, we let $G[A]$ denote the \emph{induced subgraph} of $G$ on vertex set $A$. In other words, $G[A]$ is the subgraph of $G$ obtained by restricting to the vertex set $A$. For weighted graphs, we define $(G,\weight)[A]$ similarly.

A \emph{real projection} of a (weighted) graph $(G,\weight)$ is a function $p: V(G) \to \R$ such that $|p(x)-p(y)|\le \dist{}{x,y}$ for every pair of vertices $x,y\in V(G)$.
Given $S>0$ and a real projection $p$ of a weighted graph $G$, we say that a vertex set $A\subseteq V(G)$ is \emph{$(p,S)$-bounded} if $\max_{x,y\in A}|p(x)-p(y)| \le S$.
Given a class $\mathcal{C}$ of weighted graphs and a sequence $\mathcal{L}=(\mathcal{L}^1, \mathcal{L}^2. \ldots )$ of classes of weighted
graphs with $\mathcal{L}^1 \subseteq \mathcal{L}^2 \subseteq \cdots $, we say that $\mathcal{C}$ is \emph{$\mathcal{L}$-layerable} if there is a function $f:\R^+ \to \N$ such that any weighted graph $(G,\weight) \in \mathcal{C}$ has a real projection $p: V(G) \to \R$ such that for any $S > 0$, any maximal $(p,S)$-bounded set in $(G,\weight)$ induces a (weighted) subgraph that belongs to $\mathcal{L}^{f(S)}$.

\begin{theorem}[{\cite[Theorem 4.3]{bonamy2023asymptotic}}]\label{thm:layer}
     Let $\mathcal{L} =(\mathcal{L}_1, \mathcal{L}_2, \ldots )$ be a sequence of classes of weighted graphs of asymptotic dimension at most $n$. Let $\mathcal{C}$ be an $\mathcal{L}$-layerable class of weighted graphs.
     Then the asymptotic dimension of $\mathcal{C}$ is at most $n + 1$.
\end{theorem}

We use the well-known fact that quasi-isometric classes of graphs have equal asymptotic dimension.
A \emph{$(\Pi,\Sigma)$-quasi-isometry} between (weighted) graphs $(G,\weight)$ and $(H, \psi)$ is a map $f : V(G) \to V(H)$ such that the following hold for fixed constants $\Pi \ge 1$, $\Sigma \ge 0$:
\begin{enumerate}
    \item $\Pi^{-1}\dist{}{x, y} - \Sigma \le \dist{}{f(x), f(y)} \le \Pi \dist{}{x, y} + \Sigma$ for every $x, y \in V(G)$, and
    \item for every $z \in V(H)$, there is a $x\in V(G)$ such that $\dist{}{z, f(x)} \le \Sigma$.
\end{enumerate}
We say that classes of weighted graphs $\mathcal{G}$ and $\mathcal{H}$ are \emph{quasi-isometric}, if there exist fixed constants $\Pi \ge 1$, $\Sigma \ge 0$ such that for every $G\in \mathcal{G}$ there exists a $H\in \mathcal{H}$ and a \emph{$(\Pi,\Sigma)$-quasi-isometry} between $G$ and $H$.
It is well known that quasi-isometric classes of (weighted) graphs have equal asymptotic dimension (see for instance \cite{bonamy2023asymptotic}).

\begin{lemma}\label{lem:QI}
    If classes of weighted graphs $\mathcal{G}_1, \mathcal{G}_2$ are quasi-isometric, then they have equal asymptotic dimension.
\end{lemma}

Instead of working with $\IntGraphs{d}{\alpha}$, we shall define an essentially equivalent (we only add edges of weight 2 between some vertices at distance 2) class of weighted graphs $\WeightedIntGraphs{d}{\alpha}$ that is layerable by a class of weighted graphs that are quasi-isometric to graphs in $\CompIntGraphs{d}{\alpha}$.

For $G\in \IntGraphs{d}{\alpha}$, choose some maximal $v\in V(G)$.
We call $v$ the \emph{pivot} vertex of $G$.
Now, for each non-negative integer $i$, let $D_i$ be the vertices in $G$ at distance $i$ from $v$.
Let $(\hat{G},\weight)$ be the weighted graph obtained from $G$ by adding, for each pair $x,y\in V(G)$ with $x,y\in D_i$ for some $i$, an edge of weight 2 between $x$ and $y$ if $x$ is contained in $y$ (all the edges $(E(G))$ have weight 1 in $(\hat{G},\weight)$).
The following observation is key.

\begin{lemma}\label{lem:dequal}
    Let $G\in \IntGraphs{d}{\alpha}$. Then for every pair $x,y\in V(G)$ holds ${\dist{G}{x,y} = \dist{(\hat{G},\weight)}{x,y}}$.
\end{lemma}

\begin{proof}
    It is enough to show that if there is an edge of weight 2 between vertices $x,y$ in $(\hat{G},\weight)$, then in $G$, there is a vertex adjacent to both $x$ and $y$.
    Since there is an edge of weight 2 between vertices $x,y$ in $(\hat{G},\weight)$, there exists some positive integer $i$ such that $x,y\in D_i$, and without loss of generality, we may assume that $x$ is contained in $y$.
    Let $P$ be a path of length $i$ between $x$ and $v$, and let $z$ be the unique vertex of $D_{i-1}\cap V(P)$.
    Then $z$ intersects $x$ and $\bigcup_{u\in V(P)} u$ contains a curve $C$ from $x$ to a point in $v$ not contained in $w$ for any $w\in V(G)$.
    Since the curve $C$ starts in the interior of $y$ and ends in its exterior, $C$ must intersect $y$. Therefore, $y$ is adjacent to a vertex of $P$. However, $y$ is non-adjacent to $x$ in $G$ and cannot be adjacent to any vertex of $D_0 \cup \cdots \cup D_{i-2}$ since $y\in D_i$.
    Therefore, $y$ must be adjacent to $z$ in $G$.
    Hence, the distance between $x$ and $y$ in $G$ is equal to 2, as desired.
\end{proof}

Let $\mathcal{L}_{d,\alpha}^i$ be the class of weighted graphs that are $(i,1)$-quasi-isometric to a graph in $\CompIntGraphs{d}{\alpha}$. 
Note that $\mathcal{L}_{d,\alpha}^1 \subseteq  \mathcal{L}_{d, \alpha}^2 \subseteq \cdots$.
Let $\mathcal{L}_{d, \alpha} = (\mathcal{L}_{d,\alpha}^1 , \mathcal{L}_{d, \alpha}^2 , \ldots )$.
We show that $\WeightedIntGraphs{d}{\alpha}$ is $\mathcal{L}_{d, \alpha}$-layerable.

\begin{lemma}\label{lem:layerable}
    The class of weighted graphs $\WeightedIntGraphs{d}{\alpha}$ is $\mathcal{L}_{d, \alpha}$-layerable.
\end{lemma}

\begin{proof}
    Let $(\hat{G},\weight)\in \WeightedIntGraphs{d}{\alpha}$, let $G \in \IntGraphs{d}{\alpha}$ be the graph obtained from $(\hat{G},\weight)$ by removing its edges of weight 2, and let $v$ be the pivot-vertex of $(\hat{G},\weight)$.
    For $x\in V(\hat{G})$, let $p(x)\coloneqq d_{\hat{G}}(v,x)$.
    Then, $p$ is a real projection of $(\hat{G},\weight)$.
    Let $S>0$, and let $A\subseteq V(\hat{G})$ be some maximal $(p,S)$-bounded set in $(\hat{G},\weight)$.
    Then, there exists some non-negative integer $t$ such that $A$ is the set of vertices in $(\hat{G},\weight)$ (or equivalently $G$) whose distance from $v$ is between $t$ and $t + \lceil S \rceil$; indeed, we can let $t\coloneqq \min \{p(x) \mid x\in A\}$.
    
    Let $M$ be the set of maximal vertices of $G[A]$.
    Observe that $(\hat{G},\weight)[M]=G[M]$ contains no edge of weight 2 and furthermore that it is a graph in $\CompIntGraphs{d}{\alpha}$. 
    Let $f: A \to M$ be such that for every $u\in A$, $f(u)$ is some maximal vertex containing $u$.
    We shall show that $f$ is a $(2\lceil S \rceil + 4, 1)$-quasi-isometry from $(\hat{G},\weight)[A]$ to $G[M]$, which implies that $(\hat{G},\weight)[A]\in \mathcal{L}_{d,\alpha}^{2\lceil S \rceil+4}$ and therefore that $\WeightedIntGraphs{d}{\alpha}$ is $\mathcal{L}_{d, \alpha}$-layerable as we desire.
    
    If there is an edge of weight 1 or 2 between vertices $x$ and $y$ of $(\hat{G},\weight)[A]$, then $f(x)$ and $f(y)$ intersect or coincide, and therefore $\dist{G[M]}{f(x),f(y)}\le 1$.
    It follows that $\dist{G[M]}{f(x),f(y)} \le \dist{(\hat{G},\weight)[A]}{x,y}$ for every pair of vertices $x,y\in A$.
    
    Consider some $m\in M$. Note that $f^{-1}(m)$ is non-empty since $f(m)=m$, and so $f$ is surjective. We aim to show that each vertex $x$ of $f^{-1}(m)$ has distance at most $S+2$ from $m$ in $(\hat{G},\weight)[A]$. 
    Let $x_0x_1\ldots x_\ell$ be a $v$--$x$~geodesic in $(\hat{G},\weight)$ (or $G$), i.e.\ a  path  with $x_0=v$, $x_\ell=x$ and $p(x_i) = \dist{(\hat{G},\weight)}{v,x_i}=i$ for each $i$.
    Then $\bigcup_{i=0}^\ell x_i$ contains a curve $C$ from $x$ to a point in $v$ not contained in the interior of $z$ for any $z\in V(G)\backslash v$.
    Since the curve $C$ starts in the interior of $\closure{m}$ and ends outside the interior of $\closure{m}$, it follows that $m$ is adjacent in $G$ to some vertex of $\{x_0, \ldots, x_\ell$\}.
    We further have that $m$ must be adjacent to one of $x_{p(m)-1}, x_{p(m)} , x_{p(m)+1}$.
    Note that $t\le p(m) \le p(x)=\ell \le t + \lceil S \rceil$ since $m,x\in A$.
    If either $p(m)>t$ or $m$ is adjacent to one of $x_{p(m)}$ or $x_{p(m)+1}$, then there is a path in $G[A]$ (and thus in $(\hat{G},\weight)[A]$) between $m$ and $x$ of length at most $\lceil S \rceil +1$.
    So, we may assume now that $p(m)=t$ and that $m$ is non-adjacent to both $x_{p(m)}$ and $x_{p(m)+1}$.
    It follows that $x_{p(m)}$ is contained in $\closure{m}$ and that $m$ is adjacent to $x_{p(m)-1}$ in $G$.
    Therefore, there is an edge of weight 2 in $(\hat{G},\weight)$ between $m$ and $x_{p(m)}$.
    Hence, there is a path of length at most $\lceil S \rceil +2$ in $(\hat{G},\weight)[A]$ between $m$ and $x$.
    Thus, $x$ has distance at most $\lceil S \rceil +2$ from $m$ in~$(\hat{G},\weight)[A]$.
    
    So $f^{-1}(m)$ has radius at most $\lceil S \rceil+2$ in $(\hat{G},\weight)[A]$ and therefore diameter at most $2\lceil S \rceil+4$.
    Thus, as $G[M]$ is a graph, for every $x,y\in A$, we have that $(2\lceil S \rceil+4)^{-1}\dist{(\hat{G},\weight)[A]}{x,y} - 1 \le \dist{G[M]}{f(x), f(y)}$.
    Hence $f$ is a $(2\lceil S \rceil+4)$-quasi-isometry between $(\hat{G},\weight)[A]$ and $G[M]$.
\end{proof}

We have assembled all the ingredients to prove the main result of this section:
\begin{theorem} \label{thm asdim}
    For every positive integer $d$ and every $\alpha \ge 1$, the class $\IntGraphs{d}{\alpha}$ of intersection graphs of systems of path-connected sets in $\R^d$, each of which is obtained from a compact convex set of aspect ratio at most $\alpha$ by removing some subset of its interior, has asymptotic dimension at most $2d + 2$.
\end{theorem}

\begin{proof}
    By \cref{lem:dequal}, the asymptotic dimension of $\IntGraphs{d}{\alpha}$ and $\WeightedIntGraphs{d}{\alpha}$ is equal.
    By \cref{lem:layerable}, $\WeightedIntGraphs{d}{\alpha}$ is $\mathcal{L}_{d,\alpha}$-layerable.
    The graph classes $\mathcal{L}_{d,\alpha}^1 , \mathcal{L}_{d, \alpha}^2 , \ldots $ all have asymptotic dimension at most $2d+1$ by \cref{lem:QI} and \cref{thm:convexasdim}.
    Therefore, by \cref{thm:layer}, $\WeightedIntGraphs{d}{\alpha}$ has asymptotic dimension at most $2d+2$.
    Hence $\IntGraphs{d}{\alpha}$ has asymptotic dimension at most $2d+2$.
\end{proof}

\section{Final remarks and open problems} \label{sec probs}

As we have seen, every graph $G$ with at least three vertices has sphere dimension at most $|V(G)|-1$ by \cref{obs every}. Can this be improved?
Providing a bound with respect to the number of edges is also interesting.

\begin{problem}
What is the maximum sphere dimension of a graph on $n$ vertices?
\end{problem}

In light of other known separator theorems for ball intersection graphs~\cite{MiTeThVaGeo}, it seems likely that the exponent in \cref{thm seps} is not optimal.
We conjecture the following.

\begin{conjecture} For any integers $t,d \geq 1$, every $n$-vertex graph $G\in\SphereInt{d}_t$ has a balanced separator of size ${\bigO{}}_{t,d}(n^{1-\frac{1}{d+1}})$.
\end{conjecture}

\noindent This conjecture is true for $d=1$ since every $n$-vertex circle graph with a forbidden $K_{t,t}$ subgraph has a separator of size $\bigO{}_t(\sqrt{n})$ due to~\cite{LeeSeparators} and \cite[Theorem~3]{foxPach2010}.

It is also likely possible to improve the asymptotic dimension bound in \cref{thm asdim} for graphs of bounded sphere dimension as follows.

\begin{conjecture}\label{conj:asdim}
    For any integer $d\geq 1$, the class of graphs with sphere dimension at most $d$ has asymptotic dimension at most $d$.
\end{conjecture}

This conjecture would be tight since the hypercubic lattice $\Z^d$ has asymptotic dimension and sphere dimension equal to $d$.
For the latter claim, the sphere dimension is at most $d$ because of the standard sphere packing centred at the integers. The fact that the sphere dimension is at least $d$ follows from a deep result by Benjamini and Schramm~\cite{BeSchrHar, BeSchrLac} stating that $\Z^{d} $ cannot be quasi-sphere packed in $\R^{d-1}$, with a little additional work.
The case $d=2$ of the conjecture (furthermore for Assouad--Nagata dimension and in the more general class of string graphs) will be proved in separate work \cite{stringQI}.
It is also easy to deduce the $d=1$ case for circle graphs from known results \cite{chepoi2012constant,FujPapAsy}.
Recently Liu and Norin \cite{liu} independently, using different methods, proved an improved asymptotic dimension bound of $d+1$.
It might be possible to reduce \cref{conj:asdim} via quasi-isometry to the class of intersection graphs of balls in $\mathbb{R}^d$ (this is the case for $d=2$~\cite{stringQI}).

\begin{conjecture}\label{conj:qi}
    For any integer $d\geq 2$, the class of graphs with sphere dimension at most $d$ is quasi-isometric to the class of intersection graphs of balls in $\mathbb{R}^d$.
\end{conjecture}

\medskip
Recall that an \emph{induced minor} of a graph is obtained by deleting vertices and contracting edges. Let $\mathrm{Forb_{ind}}(H)$ denote the class of graphs that do not contain a graph $H$ as an induced minor.

\begin{problem} \label{prob forb H}
Do the graphs in $\mathrm{Forb_{ind}}(H)$ have bounded sphere dimension for every finite graph $H$? If not, is each graph in  $\mathrm{Forb_{ind}}(H)$ uniformly quasi-isometric to a graph with bounded sphere dimension?
\end{problem}

The same questions can be asked for minors instead of induced minors. These two versions of \cref{prob forb H} are related via the following problem, which is a special case of \cite[Conjecture~1.1]{GeoPapMin} (which has been disproven \cite{DHIM}, but special cases remain open): 
\begin{problem} \label{prob ind min}
For every finite graph $H$, each graph in $\mathrm{Forb_{ind}}(H)$ is uniformly 
quasi-isometric to a graph with no $H$ minor. 
\end{problem}

\medskip

For $d\ge 3$, it is not possible to replace asymptotic dimension by Assouad--Nagata dimension in \cref{thm asdim}, even for the subclass of ball intersection graphs.
This was observed by Dvořák and Norin~\cite{Dvorak2021SSSconvexshapes}. 
Assouad--Nagata dimension is closely related to a concept in computer science called sparse partition schemes. Roughly speaking, if a graph class has bounded Assouad--Nagata dimension, and the colourings witnessing this can be computed, then certain universal versions of classical optimisation problems such as TSP, Steiner Tree and Set Cover allow approximations of bounded ratio. These approximate solutions are computed using sparse partition schemes; see, for instance,~\cite{bonamy2023asymptotic, JLNRS, BuLaTiSpa}. As mentioned above, graphs of bounded sphere dimension have unbounded Assouad--Nagata dimension, and so we cannot hope for sparse partition schemes. Still, it would be interesting to obtain bounds on these parameters as a function of the size of the graph. Our proof of \cref{thm asdim} utilises \cref{thm:layer}, which has a strengthening preserving boundedness of Assouad--Nagata dimension. Thus, the bottleneck lies in the proof of \cref{thm:convexasdim} in \cite{Dvorak2021SSSconvexshapes}. 

\section*{Acknowledgements}

This research was initiated during the Banff workshop `New Perspectives in Colouring and Structure (24w5272)'. We thank Alex Scott for the construction in \cref{thm not d}. We thank Louis Esperet for several helpful discussions.

\bibliographystyle{plainurl}
\bibliography{literature}

\end{document}

%% file: figures/case-1-figure.tex
\begin{tikzpicture}

            \begin{pgfonlayer}{foreground}
            \node (center) at (0,0) {};
            
            \node[vertex] (t-1) at ($(center)+(0,2)$) {};
            \node[vertex] (t-2) at ($(t-1)+(200:1)$) {};
            \node[vertex] (t-3) at ($(t-1)+(340:1)$) {};
            \node (t-4) at ($(t-2)+(220:0.8)$) {};
            \node[vertex] (t-5) at ($(t-4)+(360:1)$) {};
            \node[vertex] (t-6) at ($(t-5)+(360:1.2)$) {};
            \node[vertex] (t-7) at ($(t-6)+(360:1)$) {};
            \node[vertex,CornflowerBlue] (t-8) at ($(t-4)+(250:0.5)$) {};
            \node[vertex] (t-9) at ($(t-4)+(290:0.5)$) {};
            \node[vertex,CornflowerBlue] (t-10) at ($(t-5)+(250:0.5)$) {};
            \twocolouredVertex{t-11}{$(t-5)+(290:0.5)$}{LavenderMagenta}{CornflowerBlue}
            \node[vertex,CornflowerBlue] (t-12) at ($(t-6)+(250:0.5)$) {};
            \node[vertex] (t-13) at ($(t-6)+(290:0.5)$) {};
            \node[vertex,CornflowerBlue] (t-14) at ($(t-7)+(250:0.5)$) {};
            \node[vertex,CornflowerBlue] (t-15) at ($(t-7)+(290:0.5)$) {};

            \node (phi-x-label) at ($(t-3)+(80:0.5)$) {$\phi(x)$};

            \node[vertex,LavenderMagenta] (x-1) at ($(t-4)$) {};
            \node[LavenderMagenta,scale=0.7] (x-1-label) at ($(x-1)+(170:0.3)$) {$X_1$};
            \node[vertex,LavenderMagenta] (x-2) at ($(t-5)+(240:0.2)$) {};
            \node[LavenderMagenta,scale=0.7] (x-2-label) at ($(x-2)+(170:0.3)$) {$X_6$};
            \node (x-3) at ($(t-11)$) {};
            \node[LavenderMagenta,scale=0.7,align=center] (x-3-label) at ($(x-3)+(10:0.3)$) {
            $X_3$};
            \node[vertex,LavenderMagenta] (x-4) at ($(t-3)+(230:0.3)$) {};
            \node[LavenderMagenta,scale=0.7] (x-4-label) at ($(x-4)+(170:0.3)$) {$X_4$};
            \node[vertex,LavenderMagenta] (x-5) at ($(t-7)+(240:0.2)$) {};
            \node[LavenderMagenta,scale=0.7] (x-5-label) at ($(x-5)+(170:0.3)$) {$X_5$};
            \node[vertex,LavenderMagenta] (x-6) at ($(t-7)+(300:0.2)$) {};
            \node[LavenderMagenta,scale=0.7] (x-6-label) at ($(x-6)+(10:0.3)$) {$X_2$};
            \end{pgfonlayer}

            \begin{pgfonlayer}{main}
            
                \draw[edge] (t-1.center) to[curve through={(t-2) (t-4)}] (t-8);
                \draw[edge] (t-1.center) to[curve through={(t-3) (t-7)}] (t-15);
                \draw[edge] (t-4.center) to[curve through={($(t-4)!0.5!(t-9) + (45:0.05)$)}] (t-9);
                \draw[edge] (t-2.center) to[curve through={(t-5)}] (t-11);
                \draw[edge] (t-5.center) to[curve through={(x-2)}] (t-10);
                \draw[edge] (t-3.center) to[curve through={(x-4) (t-6)}] (t-12);
                \draw[edge] (t-6.center) to[curve through={($(t-6)!0.5!(t-13) + (45:0.05)$)}] (t-13);
                \draw[edge] (t-7.center) to[curve through={(x-5)}] (t-14);
            \end{pgfonlayer}

            \node (phi-y-1) at ($(center)+(-4,-1)$) {};
            \node (phi-y-2) at ($(phi-y-1)+(1.6,-0.4)$) {};
            \node (phi-y-3) at ($(phi-y-1)+(3.2,0)$) {};
            \node (phi-y-4) at ($(phi-y-3)+(1.6,-0.4)$) {};
            \node (phi-y-5) at ($(phi-y-3)+(3.2,0)$) {};
            \node (phi-y-6) at ($(phi-y-5)+(1.6,-0.4)$) {};

            \begin{pgfonlayer}{background}
                \draw[edge,AppleGreen,line width=2.5pt,out=237,in=80] (x-1) to (t-8);
                \draw[edge,AppleGreen,line width=2pt,out=237,in=80] (t-8) to ($($(phi-y-1)+(150:0.6)$)!0.3!($(phi-y-1)+(30:0.6)$)$);
                \draw[edge,AppleGreen,line width=2pt,out=237,in=80] (t-10) to ($($(phi-y-2)+(150:0.6)$)!0.7!($(phi-y-2)+(30:0.6)$)$);
                \draw[edge,AppleGreen,line width=2pt,out=60,in=120] ($($(phi-y-1)+(150:0.6)$)!0.6!($(phi-y-1)+(30:0.6)$)$) to ($($(phi-y-2)+(150:0.6)$)!0.3!($(phi-y-2)+(30:0.6)$)$);
                \draw[edge,AppleGreen,line width=2.5pt,out=250,in=85] (x-2) to (t-10);
                \draw[edge,AppleGreen,line width=1.7pt,out=270,in=100] ($($(phi-y-1)+(150:0.6)$)!0.3!($(phi-y-1)+(30:0.6)$)$) to ($(phi-y-1)+(240:0.2)$);
                \draw[edge,AppleGreen,line width=1.7pt,out=270,in=80] ($($(phi-y-1)+(150:0.6)$)!0.6!($(phi-y-1)+(30:0.6)$)$) to ($(phi-y-1)+(240:0.2)$);
                \draw[edge,AppleGreen,line width=1.7pt,out=270,in=100] ($($(phi-y-2)+(150:0.6)$)!0.3!($(phi-y-2)+(30:0.6)$)$) to ($(phi-y-2)+(280:0.2)$);
                \draw[edge,AppleGreen,line width=1.7pt,out=270,in=80] ($($(phi-y-2)+(150:0.6)$)!0.7!($(phi-y-2)+(30:0.6)$)$) to ($(phi-y-2)+(280:0.2)$);

                \draw[edge,AppleGreen,line width=2pt,out=275,in=80] (t-11.center) to ($($(phi-y-3)+(150:0.6)$)!0.3!($(phi-y-3)+(30:0.6)$)$);
                \draw[edge,AppleGreen,line width=2pt,out=280,in=80] (t-12) to ($($(phi-y-4)+(150:0.6)$)!0.7!($(phi-y-4)+(30:0.6)$)$);
                \draw[edge,AppleGreen,line width=2pt,out=60,in=120] ($($(phi-y-3)+(150:0.6)$)!0.6!($(phi-y-3)+(30:0.6)$)$) to ($($(phi-y-4)+(150:0.6)$)!0.3!($(phi-y-4)+(30:0.6)$)$);
                \draw[edge,AppleGreen,line width=2.5pt,out=250,in=70] (x-4) to (t-12);
                \draw[edge,AppleGreen,line width=1.7pt,out=270,in=100] ($($(phi-y-3)+(150:0.6)$)!0.3!($(phi-y-3)+(30:0.6)$)$) to ($(phi-y-3)+(290:0.2)$);
                \draw[edge,AppleGreen,line width=1.7pt,out=270,in=80] ($($(phi-y-3)+(150:0.6)$)!0.6!($(phi-y-3)+(30:0.6)$)$) to ($(phi-y-3)+(290:0.2)$);
                \draw[edge,AppleGreen,line width=1.7pt,out=270,in=100] ($($(phi-y-4)+(150:0.6)$)!0.3!($(phi-y-4)+(30:0.6)$)$) to ($(phi-y-4)+(280:0.1)$);
                \draw[edge,AppleGreen,line width=1.7pt,out=270,in=80] ($($(phi-y-4)+(150:0.6)$)!0.7!($(phi-y-4)+(30:0.6)$)$) to ($(phi-y-4)+(280:0.1)$);

                \draw[edge,AppleGreen,line width=2.5pt,out=250,in=85] (x-5) to (t-14);
                \draw[edge,AppleGreen,line width=2pt,out=275,in=100] (t-14) to ($($(phi-y-5)+(150:0.6)$)!0.3!($(phi-y-5)+(30:0.6)$)$);
                \draw[edge,AppleGreen,line width=2pt,out=280,in=90] (t-15) to ($($(phi-y-6)+(150:0.6)$)!0.7!($(phi-y-6)+(30:0.6)$)$);
                \draw[edge,AppleGreen,line width=2pt,out=60,in=120] ($($(phi-y-5)+(150:0.6)$)!0.6!($(phi-y-5)+(30:0.6)$)$) to ($($(phi-y-6)+(150:0.6)$)!0.3!($(phi-y-6)+(30:0.6)$)$);
                \draw[edge,AppleGreen,line width=2.5pt,out=280,in=100] (x-6) to (t-15);
                \draw[edge,AppleGreen,line width=1.7pt,out=270,in=100] ($($(phi-y-5)+(150:0.6)$)!0.3!($(phi-y-5)+(30:0.6)$)$) to ($(phi-y-5)+(250:0.2)$);
                \draw[edge,AppleGreen,line width=1.7pt,out=270,in=80] ($($(phi-y-5)+(150:0.6)$)!0.6!($(phi-y-5)+(30:0.6)$)$) to ($(phi-y-5)+(250:0.2)$);
                \draw[edge,AppleGreen,line width=1.7pt,out=270,in=100] ($($(phi-y-6)+(150:0.6)$)!0.3!($(phi-y-6)+(30:0.6)$)$) to ($(phi-y-6)+(20:0.1)$);
                \draw[edge,AppleGreen,line width=1.7pt,out=270,in=80] ($($(phi-y-6)+(150:0.6)$)!0.7!($(phi-y-6)+(30:0.6)$)$) to ($(phi-y-6)+(20:0.1)$);

            \end{pgfonlayer}

            \begin{pgfonlayer}{deepbackground}
                \foreach \i in {1,...,6}
                {
                    \fill[PastelOrange!60,rounded corners] ($(phi-y-\i)+(150:0.6)$) -- ($(phi-y-\i)+(30:0.6)$) -- ($(phi-y-\i)+(270:0.8)$) -- cycle;
                }
                \node[PastelOrange,scale=0.7] (phi-y-1-label) at ($(phi-y-1)+(270:1)$) {$\phi(y(X_{1}))$};
                \node[PastelOrange,scale=0.7] (phi-y-2-label) at ($(phi-y-2)+(270:1)$) {$\phi(y(X_{6}))$};
                \node[PastelOrange,scale=0.7] (phi-y-3-label) at ($(phi-y-3)+(270:1)$) {$\phi(y(X_{3}))$};
                \node[PastelOrange,scale=0.7] (phi-y-4-label) at ($(phi-y-4)+(270:1)$) {$\phi(y(X_{4}))$};
                \node[PastelOrange,scale=0.7] (phi-y-5-label) at ($(phi-y-5)+(270:1)$) {$\phi(y(X_{5}))$};
                \node[PastelOrange,scale=0.7] (phi-y-6-label) at ($(phi-y-6)+(270:1)$) {$\phi(y(X_{2}))$};
            \end{pgfonlayer}

        \end{tikzpicture}

%% file: figures/case-2-figure.tex
\begin{tikzpicture}
    \node (centre) at (0,0) {};
    \begin{pgfonlayer}{deepbackground}
    \foreach \i in {1,...,6}
    {
        \node[vertex] (phi-\i) at ($(centre)+({360/6*\i-90}:2)$) {};
        \fill[PastelOrange!60,rounded corners] ($(phi-\i)+({360/6*\i-210}:0.7)$) -- ($(phi-\i)+({360/6*\i+30}:0.7)$) -- ($(phi-\i)+({360/6*\i-90}:1)$) -- cycle;
    }
    \end{pgfonlayer}

	\node[vertex,scale=0.8] (t-1-1) at ($(phi-1)+(330:0.6)$) {};
	\node[vertex,scale=0.8] (t-1-2) at ($(phi-1)+(450:0.4)$) {};
	\node[vertex,scale=0.8] (t-1-4) at ($(phi-1)+(210:0.4)$) {};
	\node[vertex,scale=0.8] (t-1-3) at ($(t-1-2)!0.33!(t-1-4)$) {};
	\node[vertex,scale=0.8] (t-1-5) at ($(t-1-2)!0.66!(t-1-4)$) {};
	\node[vertex,scale=0.8] (t-1-6) at ($(t-1-1)!0.66!(t-1-2)$) {};
	\node[vertex,CornflowerBlue] (t-1-7) at ($(t-1-1)!0.33!(t-1-2)$) {};

    \begin{pgfonlayer}{background}
    	\draw[edge] (t-1-1) to (t-1-2);
    	\draw[edge] (t-1-1) to (t-1-4);
    	\draw[edge] (t-1-6) to (t-1-3);
    	\draw[edge] (t-1-7) to (t-1-5);
    \end{pgfonlayer}
	
	\node[vertex,scale=0.8] (t-2-1) at ($(phi-2)+(30:0.6)$) {};
	\node[vertex,scale=0.8] (t-2-2) at ($(phi-2)+(150:0.4)$) {};
	\node[vertex,scale=0.8] (t-2-3) at ($(phi-2)+(270:0.4)$) {};
	\node[vertex,scale=0.8] (t-2-4) at ($(t-2-1)!0.6!(t-2-2)$) {};
	\node[vertex,LavenderMagenta] (t-2-5) at ($(t-2-1)!0.6!(t-2-3)$) {};
	\node[vertex,scale=0.8] (t-2-6) at ($(t-2-2)!0.33!(t-2-3)$) {};
	\node[vertex,scale=0.8] (t-2-7) at ($(t-2-2)!0.66!(t-2-3)$) {};

    \begin{pgfonlayer}{background}
    	\draw[edge] (t-2-1) to (t-2-2);
    	\draw[edge] (t-2-1) to (t-2-3);
    	\draw[edge] (t-2-4) to (t-2-6);
    	\draw[edge] (t-2-5) to (t-2-7);
    \end{pgfonlayer}
	
	\node[vertex,scale=0.8] (t-3-1) at ($(phi-3)+(90:0.6)$) {};
	\node[vertex,scale=0.8] (t-3-2) at ($(phi-3)+(210:0.4)$) {};
	\node[vertex,scale=0.8] (t-3-3) at ($(phi-3)+(330:0.4)$) {};
	\node[vertex,LavenderMagenta] (t-3-4) at ($(phi-3)+(90:0.1)$) {};
	\node[vertex,scale=0.8] (t-3-5) at ($(t-3-2)!0.33!(t-3-3)$) {};
	\node[vertex,scale=0.8] (t-3-6) at ($(t-3-2)!0.66!(t-3-3)$) {};

    \begin{pgfonlayer}{background}
    	\draw[edge] (t-3-1) to (t-3-2);
    	\draw[edge] (t-3-1) to (t-3-3);
    	\draw[edge] (t-3-1) to (t-3-4);
    	\draw[edge] (t-3-4) to (t-3-5);
    	\draw[edge] (t-3-4) to (t-3-6);
    \end{pgfonlayer}
	
	\node[vertex,scale=0.8] (t-4-1) at ($(phi-4)+(150:0.6)$) {};
	\node[vertex,scale=0.8] (t-4-2) at ($(phi-4)+(270:0.4)$) {};
	\node[vertex,scale=0.8] (t-4-3) at ($(phi-4)+(30:0.4)$) {};
	\node[vertex,scale=0.8] (t-4-4) at ($(t-4-2)!0.33!(t-4-3)$) {};
	\node[vertex,LavenderMagenta] (t-4-5) at ($(t-4-2)!0.66!(t-4-3)$) {};

    \begin{pgfonlayer}{background}
    	\draw[edge] (t-4-1) to (t-4-2);
    	\draw[edge] (t-4-1) to (t-4-3);
    	\draw[edge] (t-4-1) to (t-4-4);
    	\draw[edge] (t-4-1) to (t-4-5);
    \end{pgfonlayer}
	
	\node[vertex,CornflowerBlue] (t-5-1) at ($(phi-5)+(210:0.6)$) {};
	\node[vertex,scale=0.8] (t-5-2) at ($(phi-5)+(330:0.4)$) {};
	\node[vertex,scale=0.8] (t-5-3) at ($(phi-5)+(90:0.4)$) {};
	\node[vertex,scale=0.8] (t-5-4) at ($(t-5-2)!0.33!(t-5-3)$) {};
	\node[vertex,scale=0.8] (t-5-5) at ($(t-5-2)!0.66!(t-5-3)$) {};
	\node[vertex,scale=0.8] (t-5-6) at ($(t-5-1)!0.6!(t-5-3)$) {};

    \begin{pgfonlayer}{background}
    	\draw[edge] (t-5-1) to (t-5-2);
    	\draw[edge] (t-5-1) to (t-5-3);
    	\draw[edge] (t-5-1) to (t-5-4);
    	\draw[edge] (t-5-6) to (t-5-5);
    \end{pgfonlayer}
	
	\node[vertex,scale=0.8] (t-6-1) at ($(phi-6)+(270:0.6)$) {};
	\node[vertex,scale=0.8] (t-6-2) at ($(phi-6)+(30:0.4)$) {};
	\node[vertex,scale=0.8] (t-6-3) at ($(phi-6)+(150:0.4)$) {};
	\node[vertex,scale=0.8] (t-6-4) at ($(t-6-2)!0.33!(t-6-3)$) {};
	\node[vertex,scale=0.8] (t-6-5) at ($(t-6-2)!0.66!(t-6-3)$) {};
	\node[vertex,scale=0.8] (t-6-6) at ($(t-6-1)!0.4!(t-6-2)$) {};
	\node[vertex,CornflowerBlue] (t-6-7) at ($(t-6-6)!0.5!(t-6-5)$) {};

    \begin{pgfonlayer}{background}
    	\draw[edge] (t-6-1) to (t-6-2);
    	\draw[edge] (t-6-1) to (t-6-3);
    	\draw[edge] (t-6-6) to (t-6-5);
    	\draw[edge] (t-6-7) to (t-6-4);
    \end{pgfonlayer}

    \begin{pgfonlayer}{deepbackground}
        \draw[edge,AppleGreen,line width=1.7pt,out=330,in=40] (t-4-5) to (t-5-5);
        \draw[edge,AppleGreen,line width=3pt] (t-5-5) to (t-5-6);
        \draw[edge,AppleGreen,line width=3pt] (t-5-6) to (t-5-1);

        \draw[edge,AppleGreen,line width=1.7pt,out=260,in=80] (t-3-2) to (t-6-4);
        \draw[edge,AppleGreen,line width=3pt] (t-3-2) to (t-3-1);
        \draw[edge,AppleGreen,line width=3pt] (t-3-1) to (t-3-4);
        \draw[edge,AppleGreen,line width=3pt] (t-6-4) to (t-6-7);

        \draw[edge,AppleGreen,line width=1.7pt,out=200,in=120] (t-2-6) to (t-1-3);
        \draw[edge,AppleGreen,line width=3pt] (t-2-4) to (t-2-6);
        \draw[edge,AppleGreen,line width=3pt] (t-2-4) to (t-2-1);
        \draw[edge,AppleGreen,line width=3pt] (t-2-1) to (t-2-5);
        \draw[edge,AppleGreen,line width=3pt] (t-1-3) to (t-1-6);
        \draw[edge,AppleGreen,line width=3pt] (t-1-6) to (t-1-7);
    \end{pgfonlayer}
\end{tikzpicture}